\documentclass{article}
\usepackage{amsfonts}
\usepackage{amsmath}
\usepackage{amssymb}
\usepackage{theorem}

\newtheorem{theorem}{Theorem}
\newtheorem{proposition}[theorem]{Proposition}
\newtheorem{lemma}[theorem]{Lemma} 
\newtheorem{corollary}[theorem]{Corollary}

\theorembodyfont{\rmfamily}

\newtheorem{remark}[theorem]{Remark}
\newcommand{\CAT}{{\mathrm{CAT}}}
\newcommand{\bbar}{\underline{{\mathrm B}}}
\newcommand{\ebar}{\underline{{\mathrm E}}}
\newcommand{\cu}{{\mathcal U}}
\newcommand{\rr}{{\mathbb R}}
\newcommand{\tilm}{{\widetilde M}}
\newcommand{\tiln}{{\widetilde N}}
\newcommand{\tilw}{{\widetilde W}}
\newcommand{\tilsw}{{\widetilde w}}

\newcommand{\tilt}{{\widetilde T}}

\newcommand{\tilv}{{\widetilde v}}

\newcommand{\sd}{\mathrm{Sd}}

\title{Every finite complex is the classifying space for proper bundles of a virtual Poincar\'e duality group.}

\author{Raeyong Kim}

\date{\today}

\newenvironment{proof}[1][]{\begin{trivlist} \item[\hskip\labelsep
\emph{Proof#1.}]}{\foorp \end{trivlist}}
\newcommand{\foorp}{{\unskip\nobreak\hfil\penalty50
 \hskip1em\vadjust{}\nobreak\hfil \vrule height3pt width3pt depth0pt
 \parfillskip=0pt \finalhyphendemerits=0 \par}}

\begin{document} 

\maketitle

\begin{abstract} 

We prove that every finite connected simplicial complex is homotopy equivalent to the quotient of a contractible manifold by proper actions of a virtually torsion-free group. As a corollary, we obtain that every finite connected simplicial complex is homotopy equivalent to the classifying space for proper bundles of some virtual Poincar\'e duality group.

\end{abstract} 

\section{Introduction}

Let $G$ be a discrete group. A $G$-CW-complex is, by definition, a CW-complex on which $G$ acts by permuting the cells and cell stabilizers act trivially on cells. A $G$-CW-complex $Y$ is said to be a model for $\ebar G$ if every cell stabilizer is finite and, for every finite subgroup $H \leq G$, the fixed set $Y^{H}$ is contractible. The existence of such a model can be established by Milnor's or Segal's argument for the construction of the universal space for $G$. (See \cite{LeaNuc} for the general construction). Applying an equivariant obstruction theory proves that any two models for $\ebar G$ are $G$-homotopy equivalent. We call $\ebar G$ the classifying space for proper $G$-actions. We write the quotient $\ebar G/G$ by $\bbar G$ and call it the classifying space for proper $G$-bundles. Our main theorem can be stated as follows.

\begin{theorem}\label{mainthm}
For any finite connected simplicial complex $X$, there exists a virtually torsion-free group $G$ with $\ebar G$  a cocompact manifold such that $\bbar G$ is homotopy equivalent to $X$.
\end{theorem}

A group $\Gamma$ is called \emph{a Poincar\'e duality group of dimension $n$} if $H^{i}(\Gamma;A) \cong H_{n-i}(\Gamma;A)$ for any $\mathbb{Z}\Gamma$-module $A$. Many interesting  examples of Poincar\'e duality groups are manifold groups. More specifically, the fundamental group of a closed, aspherical $n$-dimensional manifold is a Poincar\'e duality group of dimension $n$. (The converse is false.) See \cite{Bro}, \cite{DavPOINCARE} for details about Poincar\'e duality groups. Finally, recall that a group \emph{virtually} has some property if a finite index subgroup has the property.

Let $T$ be a torsion-free finite index subgroup of $G$ in Theorem \ref{mainthm}. Then one can take $\ebar G$ as a model for $\ebar T = ET$, where $ET$ is the universal space for $T$. Since $ET/T$ is a closed aspherical manifold, $T$ is a Poincar\'e duality group. 

\begin{corollary}\label{maincor}
For any finite connected simplicial complex $X$, there is a virtual Poincar\'e duality group $G$ such that $\bbar G$ is homotopy equivalent to $X$.
\end{corollary}

The statement of Corollary \ref{maincor} is related to the theorem of Kan-Thurston, which says that every connected complex has the same homology as the classifying space for some group (See \cite{KanThu}). This theorem has been extended and generalized by a number of authors. For example, see \cite{BauDyeHel},\cite{Mau},\cite{Hau},\cite{LeaMKT},\cite{McD},\cite{LeaNuc},\cite{RKim}. Among many extensions and generalizations, Leary and Nucinkis proved in \cite{LeaNuc} that every connected CW-complex has the same homotopy type as the classifying space for proper bundles of some group. Corollary \ref{maincor} says that the group can be taken as a virtual Poincar\'e duality group if the simplicial complex is finite.

The proof of Theorem \ref{mainthm} consists of three steps. In Section \ref{equiembedding}, we outline the embedding trick, due to Floyd \cite{Flo}, for equivariantly embedding a simplicial complex with an involution into some Euclidean space. In Section \ref{equireflection}, we review the equivariant reflection group trick, due to Davis \cite{DavLea}. Finally, in Section \ref{resultproof}, we use the two tricks to construct a contractible manifold, whose quotient by some group $G$ is homotopy equivalent to a given finite simplicial complex $X$. We also prove that the contractible manifold is the classifying space for proper $G$-actions and introduce the torsion-free subgroup of finite index to complete the proof of Theorem \ref{mainthm}.

The paper is part of author's Ph.D. thesis. The author thanks the thesis advisor Ian Leary for his guidance throughout this research project. The author also thanks Jean Lafont for his careful reading of an earlier version of this paper.

\section{The Equivariant Embedding Trick}\label{equiembedding}

Let $Z$ be a finite simplicial complex with a simplicial map of period $p$. In \cite{Flo}, Floyd introduced the embedding trick, namely, $Z$ can be embedded in some Euclidean space  such that the restriction of specific coordinate changing map on $Z$ is the given simplicial map. We outline his construction in the case that $p=2$. See \cite[Section 2]{Flo} for the full construction.

Let $L$ be a finite connected simplicial complex with a simplicial involution $T$. Embed $L$ into $\rr^{n}$ for some $n$ and suppose that $\rr^{n}$ is triangulated so that $L$ is a subcomplex. Consider the following map.
$$\phi : L \to  \rr^{n} \times \rr^{n}(=\rr^{2n}) ,\qquad x \mapsto (x,T(x)).$$

Note that a cell in the cellular decomposition of $\rr^{2n}$ has the type of $s_1 \times s_2$, where each $s_i$ is a simplex in $\rr^{n}$. We use the first barycentric subdivision of this cellular decomposition for the subdivision of $\rr^{2n}$. 

By passing to the barycentric subdivision $\sd(L)$ of $L$, we obtain that $\phi$ is a simplicial homeomorphism of $\sd(L)$ onto a subcomplex of $\rr^{2n}$. Furthermore, the map $S : \rr^{n} \times \rr^{n} \to \rr^{n} \times \rr^{n}$ defined by $S(x,y)=(y,x)$ is simplicial and satisfies $S\circ\phi=\phi\circ T$, hence $\phi$ is an equivalence between $(\sd(L),T)$ and $(\phi(\sd(L)),S)$. Hereafter, we suppose that $L$ is a subcomplex of $\rr^{2n}$ and $S=T$ on $L$.

We may as well assume that $L$ is a full subcomplex of $\rr^{2n}$. For if not, $\sd(L)$ is a full subcomplex of $\sd(\rr^{2n})$. Let $U$ be the first regular neighborhood of $L$, i.e. the union of all open stars of vertices of $\sd(L)$ relative to $\sd(\rr^{2n})$. Then $K=\overline{U}$ is a manifold with boundary of dimension $2n$. Denote the boundary of $K$ by $\partial K$.

Let $v_0, \cdots, v_{k}, v_{k+1}, \cdots$ be vertices of $\rr^{2n}$, where $v_0,\cdots, v_{k}$ are vertices of $L$. Every point $x$ in $\rr^{2n}$ has a unique barycentric representation $\displaystyle{\sum t_{i}v_{i}}$. Furthermore, $K$ consists of points $x$ for which $\mathrm{max}(t_0, \cdots , t_k) \geq \mathrm{max}(t_{k+1}, \cdots )$ and $\partial K$ consists of points $x$ for which $\mathrm{max}(t_0, \cdots , t_k) = \mathrm{max}(t_{k+1}, \cdots )$. Consider $f : K \to L$ defined by $\displaystyle{f(x) = f(\sum t_{i}v_{i}}) = \frac{\sum_{i=0}^{k} t_{i}v_{i}}{\sum_{i=0}^{k} t_{i}}$. Then $\Phi : K \times I \to K$ defined by $\Phi(x,t) = (1-t)x+ tf(x)$ is a deformation retract of $K$ onto $L$.

\begin{remark}
By the unique barycentric representation of points in $\rr^{2n}$, $\Phi(S(x),t) = S(\Phi(x),t)$ for any $x \in K$ and $t \in I$ so that $\Phi$ is $S$-equivariant. Therefore, $K$ equivariantly deformation retracts onto $L$. In particular, the fixed set $K^S$ is homotopy equivalent to the fixed set $L^T$ and $K/\langle S\rangle$ is homotopy equivalent to $L/\langle T \rangle$.
\end{remark}

\section{The Equivariant Reflection Group Trick}\label{equireflection}

Suppose that we are given a space $M$ and a subspace $N \subset M$ such that $N$ is triangulated as a finite dimensional flag complex. Recall that a simplicial complex is a \emph{flag complex} if any finite set of vertices, which are pairwise connected by edges, spans a simplex. Let a discrete group $G$ act on $M$ so that $G$ stabilizes the subspace $N$ and $G$ acts on $N$ by simplicial automorphisms. Following \cite{DavLea}, we will associate a right-angled Coxeter group $W$ and construct a $(W \rtimes G)$-action on a space $\cu(M,N,G)$.

Let $I$ be a vertex set of $N$. Define a right-angled Coxeter group $W$ as follows. There is one generator $s_{i}$ for each $i \in I$. Relations are given by $s_{i}^{2} = 1$ and $(s_{i}s_{j})^{2} = 1$ if $\{i,j\}$ spans an edge in $N$. For $x \in N$, let $\sigma(x) = \{i \in I | x \in N_i\}$, where $N_i$ is the closed star of the vertex $i$ in the barycentric subdivision of $N$ and $W_{x}$ be the subgroup generated by $\{s_i | i \in \sigma(x)\}$.

Note that $G$ acts on $N$ by permuting vertices, so we can form $W \rtimes G$.

Define the space $\cu(M,N,G)$ by
$$\cu(M,N,G) = W \times M / \sim,$$
where $(w,x) \sim (w',x')$ if $x=x'$ and $w^{-1}w' \in W_{x}$. For $[w,x] \in \cu(M,N,G)$ and $(v,g) \in W \rtimes G$, the action of $W\rtimes G$ on $\cu(M,N,G)$ is defined by 
$$(v,g).[w,x] = [vw^{g},g .x].$$

\begin{remark}
This construction enjoys the following properties. For details or proofs, see \cite[Section 11.7]{DavBOOK}, \cite{Dav}.
\begin{enumerate}
\item If $M$ is contractible, then so is $\cu(M,N,G)$.
\item If $M$ is an $n$-dimensional manifold with boundary and $N= \partial M$, then $\cu(M,N,G)$ is an $n$-dimensional manifold.
\item Let $C(N)$ be the cone on $N$. Then $\cu(C(N),N,G)$ has a natural $\CAT(0)$ cubical structure so that the link of each vertex is isomorphic to $N$, and so that $W \rtimes G$ acts by a group of isometries. In particular, for any finite subgroup $F$ of  $W \rtimes G$, the fixed point set $\cu(C(N),N,G)^{F}$ is contractible.
\end{enumerate}
\end{remark}

A group action on a simplicial complex is said to be \emph{admissible} if, for any simplex, setwise stabilizers are equal to pointwise stabilizers. If the $G$-action on $N$ is admissible, we have the following.

\begin{lemma}\label{finitesubgroup}
Let $H$ be a finite subgroup in $G$. Then 
$$\cu(M,N,G)^{H} = \cu(M^{H},N^{H},V_{H}),$$
where $M^{H}$, and $N^{H}$ respectively, is the $H$-fixed set in $M$, and $N$ respectively, and $V_{H} = N_{G}(H)/H$.
\end{lemma}

\begin{remark} The above lemma is stated in \cite[Proposition 11.7.1]{DavBOOK} without a proof. We provide the proof below. Note also that $N^{H}$ is a flag complex.
\end{remark}

\begin{proof}

It is obvious that\, $\cu(M^{H},N^{H},V_{H})$ is a subspace of\, $\cu(M,N,G)$. 

Let $[w,m] \in \cu(M,N,G)^{H}$ be given. In order to prove that $\cu(M,N,G)^{H}$ is contained in $\cu(M^{H},N^{H},V_{H})$, it suffices to show that $m \in M^{H}$ and $w \in W_{H}$, where $W_{H}$ is the subgroup of $W$ generated by $\{s_{i}| i \,\,\textrm{is a vertex in}\,\, N^{H}\}$. For any $h \in H$,
$$\begin{array}{cc} &(1,h).[w,m] = [w^{h},h.m] = [w,m]\\\Rightarrow&h.m = m,\qquad w^{-1}w^{h} \in W_{m}.\end{array}$$

Since $h.m=m$ for all $h \in H$, it follows that $m \in M^{H}$. We prove that $w \in W_{H}$ by induction on the length $l(w)$ of $w$. To begin with, we point out that since the $G$-action on $M$ is admissible, every generator in $W_{m}$ is fixed by $H$.  In particular, $W_{m}$ is a subgroup in $W_{H}$. Also note that $W_{m}$ is finite. (Consider the simplex containing $m$ of minimal dimension.)

Suppose that $l(w)=1$, i.e. $w=w^{-1}$ . Since $W_{m}$ is finite, $ww^{h}$ has finite order. But this happens only if two vertices corresponding to $w$ and $w^{h}$ are connected. Admissibility implies that $w$ is fixed by $h$, and hence, $w \in W_{H}$.

Suppose that $w = s_1 \cdots s_n$ is a reduced word (of length $n$). Let $t_i = s_{i}^{h}$. 

Suppose $s_{1} \neq t_{1}$. Again, $w^{-1}w^{h}$ has finite order.
$$\begin{array}{ccc}& (w^{-1}w^{h})^{n} =  1\,\, \mathrm{for\,\,some}\,\, n\\\Rightarrow&(s_n \cdots s_1\cdot t_1 \cdots t_n) (s_n \cdots s_1\cdot t_1 \cdots t_n)\cdots(s_n \cdots s_1\cdot t_1 \cdots t_n)=1\end{array}$$

In order for the left hand side to be reduced to the identity, in particular, there exists $t_{i}$ for $2 \leq i \leq n$ such that $t_{i}$ commutes with $t_{1}$ and cancels with $s_{1}$, i.e. $t_{i}=s_{1}$. But this is impossible, because $s_{1}$ and $t_{1}$ do not commute. Therefore, $s_{1}=t_{1}$.
 
$$w^{-1}w^{h} = s_n \cdots s_1\cdot t_1 \cdots t_n = s_n \cdots s_2\cdot t_2 \cdots t_n \in W_{m}$$
By the induction hypothesis, $s_{1} w \in W_{H}$, so $w \in W_{H}$. 
\end{proof}

\section{The proof of Theorem \ref{mainthm}}\label{resultproof}

The proof of Theorem \ref{mainthm} consists of three steps. First, we use the equivariant embedding trick to embed a given finite simplicial complex $X$ into the manifold $M$ with an involution $\tau$ such that $M/\langle \tau \rangle$ is homotopy equivalent to $X$. Then we apply the equivariant reflection group trick on $M$ with boundary to obtain a contractible manifold on which some group $G$ acts. The quotient of the contractible manifold by $G$ will be homotopy equivalent to $X$. Finally, we show that the contractible manifold is the classifying space for proper $G$-actions. Additionally, we introduce a finite index torsion-free subgroup of $G$, which proves that $G$ is a virtual Poincar\'e duality group.\\

Let $X$ be a finite connected simplicial complex. Note that the equivariant embedding trick requires a simplicial complex with a periodic simplicial map. In this paper, we use the construction appearing in \cite{LeaMKT}. Applying \cite[Theorem A]{LeaMKT}, we obtain a finite connected locally $\CAT(0)$ cubical complex $Y$ with a cubical involution $\tau$ such that $Y/\langle\tau\rangle$ is homotopy equivalent to $X$. By passing to the barycentric subdivision, we may assume that $Y$ is a finite connected simplicial complex  and $\tau$ is a simplicial involution on $Y$. Now we apply the equivariant embedding trick introduced in Section \ref{equiembedding} to obtain a manifold $M$ with a boundary $N$ and a simplicial involution $\omega$ on $M$ such that $M$ equivariantly deformation retracts onto $Y$. By passing to the barycentric subdivision, we can assume that $N$ is a flag complex and a cyclic group of order two, $C_{2}=\langle \omega \rangle$, acts on $N$ admissibly. Note that $Y$ is locally $\CAT(0)$. Therefore, $M$ is aspherical and $M^{\omega}$ is homotopy equivalent to $Y^{\tau}$.

Next we apply the equivariant reflection group trick from Section \ref{equireflection} to obtain a space\, $\cu(M,N,C_{2})$ on which $W \rtimes C_{2}$ acts, where $W$ is a right-angled Coxeter group associated to $N$. Let $\tilm$ be the universal cover of $M$, $\tiln$ be the inverse image of $N$ in $\tilm$ and $\tilw$ be the associated right-angled Coxeter group to $\tiln$. Repeat the equivariant reflection group trick to obtain a space\, $\cu(\tilm,\tiln,\Gamma)$ on which $\tilw \rtimes \Gamma$ acts, where $\Gamma$ is the group of liftings of the $C_{2}$-action on $M$ to $\tilm$. Note that every torsion element in $\Gamma$ has order at most two and every finite subgroup of $\Gamma$ is cyclic of order two.

\begin{proposition}
\label{equi}
Let \,$\cu(M,N,C_2)$ and \,$\cu(\tilm,\tiln,\Gamma)$ be the spaces constructed above. Then
\begin{enumerate}
\item\label{p1} $\cu(\tilm,\tiln,\Gamma)$ and $\cu(M,N,C_2)$ are manifolds.
\item\label{p2} $\cu(\tilm,\tiln,\Gamma)$ is the universal cover of $\cu(M,N,C_2)$.
\item\label{p3} $\cu(\tilm,\tiln,\Gamma)/(\tilw \rtimes \Gamma)$ is homotopy equivalent to $X$.
\end{enumerate}
\end{proposition}

\begin{proof} 
The first statement follows from the fact that $M$ and $\tilm$ are manifolds with boundary. It is clear that $\cu(\tilm,\tiln,\Gamma)$ is a cover of $\cu(M,N,C_2)$. Since $M$ is aspherical, $\tilm$ is contractible and so is $\cu(\tilm,\tiln,\Gamma)$. This proves the second statement. By construction, 

$$\cu(\tilm,\tiln,\Gamma)/(\tilw \rtimes \Gamma) \simeq \cu(M,N,C_2)/ (W\rtimes C_{2}) \simeq M/C_{2} \simeq X,$$
where $\simeq$ is a homotopy equivalence.
\end{proof}

We compare the $(\tilw \rtimes \Gamma)$-action on the space\, $\cu(\tilm,\tiln,\Gamma)$\,\, with the same action on\, $\cu(C(\tiln),\tiln,\Gamma)$, and prove $\cu(\tilm,\tiln,\Gamma) = \ebar(\tilw \rtimes \Gamma)$. Denote the image of $\{\tilsw\} \times \tilm$ in $\cu(\tilm, \tiln, \Gamma)$ by $\tilsw\tilm$ and the image of $\{\tilsw\} \times (\tilm \setminus \tiln)$ by $int(\tilsw\tilm)$. First, we prove that all stabilizers are finite.

\begin{proposition}
Let $H$ be a subgroup of  $\tilw \rtimes \Gamma$ fixing some point in $\cu(\tilm,\tiln,\Gamma)$. Then $H$ is finite.
\end{proposition}

\begin{proof}

Suppose that $H$ fixes some point in $int(\tilsw\tilm)$ for some $\tilsw$.

Then $H' = (\tilsw,1)^{-1}H(\tilsw,1)$ fixes some point in $int(\tilm)$. Denote this point by $[1,x]$. For any $(\tilv,\gamma) \in H'$,
$$(\tilv,\gamma).[1,x] = [\tilv,\gamma.x] =[1,x]  \Rightarrow  \gamma.x=x,\,\, \tilv \in \tilw_{x}$$

Since $x \in \tilm \setminus \tiln$(recall that $\tiln =\partial \tilm$, and $x \in int(\tilm)$.), $\tilw_{x}$ is trivial. It follows that $H'$ is a subgroup of $\Gamma$. Recall that $\Gamma$ is the group of liftings of the $C_{2}$-action on $M$ to $\tilm$. Therefore, if a nontrivial element $\gamma$ fixes some point $x$, $\gamma$ is the only nontrivial element in $\Gamma$ fixing $x$. This tells us that $H'$ must be finite of order $2$.

Suppose that the fixed point is not contained in $int(\tilsw\tilm)$ for any $\tilsw$. As in the previous case, choose some $\tilsw'$ so that $H''=(\tilsw',1)^{-1}H(\tilsw',1)$ fixes some point in the image of $\tiln$ in $\cu(\tilm,\tiln,\Gamma)$. Denote such a point by $[1,y]$. For any $(\tilv',\gamma') \in H''$,
$$(\tilv',\gamma').[1,y] = [\tilv',\gamma'.y] =[1,y]  \Rightarrow  \gamma'.y=y,\,\, \tilv' \in \tilw_{y}$$

As before, we have at most two possibilities for $\gamma'$. Furthermore, $\tilw_{y}$ is finite. (Consider the simplex containing $y$ of minimal dimension.) Therefore, $H''$ is finite, and hence, $H$ is finite.

\end{proof}

\begin{theorem}\label{clas}
$\cu(\tilm,\tiln,\Gamma) = \ebar(\tilw \rtimes \Gamma)$.
\end{theorem}

\begin{proof}

It suffices to prove that the fixed point set by a finite subgroup is contractible. As mentioned before, we consider  the $(\tilw \rtimes \Gamma)$-action on $\cu(C(\tiln),\tiln,\Gamma)$.

Let $F$ be a finite subgroup of $\tilw \rtimes \Gamma$. Recall that $\cu(C(\tiln),\tiln,\Gamma)^{F}$ is contractible. (See Remark 4.) In particular, it is nonempty.

Suppose that $F$ does not fix any cone point in $\cu(C(\tiln),\tiln,\Gamma)$. Then $F$ fixes no point in $int(\tilsw\tilm)$ for any $\tilsw$. In other words, 

$$\cu(\tilm,\tiln,\Gamma)^{F} \subset \displaystyle{\cu(\tilm,\tiln,\Gamma) - \bigcup_{\tilsw \in \tilw} int(\tilsw\tilm)}.$$ Therefore,

$$\cu(\tilm,\tiln,\Gamma)^{F} = \cu(C(\tiln),\tiln,\Gamma)^{F}$$
and $\cu(\tilm,\tiln,\Gamma)^{F}$ is contractible.

Suppose that $F$ fixes some cone point in $\cu(C(\tiln),\tiln,\Gamma)$. Choose some $\tilsw''$ so that $F' = (\tilsw'',1)^{-1}F(\tilsw'',1)$ fixes the cone point of $int(C\tiln)$. Denote the cone point by $c$. For any $(\tilv'',\gamma'') \in F'$,
$$(\tilv'',\gamma''),[1,c] = [\tilv'',\gamma''.c] = [1,c] \Rightarrow \gamma''.c=c,\,\,\tilv'' \in \tilw_{c}.$$

Since $c$ is a cone point, $\tilv''$ is trivial. It follows that $F'$ is a finite subgroup of $\Gamma$, and hence, cyclic of order $2$. By Lemma \ref{finitesubgroup}, it follows that $\cu(\tilm,\tiln,\Gamma)^{F'}$ is $\cu(\tilm^{F'},\tiln^{F'},V_{F'})$. Recall $M$ is equivariantly homotopy equivalent to a locally $\mathrm{CAT}(0)$ space $Y$. Therefore, $\tilm^{F'}$ is homotopy equivalent to the fixed set of $CAT(0)$ space by a cyclic group of order $2$. It follows that $\tilm^{F'}$ is contractible, hence so is $\cu(\tilm^{F'},\tiln^{F'},V_{F'})$. Finally, $\cu(\tilm,\tiln,\Gamma)^{F} = (\tilsw'',1).\cu(\tilm,\tiln,\Gamma)^{F'}$ is contractible.

\end{proof}

\begin{remark}
Consider the commutator subgroup $T$ of $W$ and its inverse image $\tilt$ in $\tilw$. Since $\tilt$ is torsion-free and of finite index in $\tilw$, $\tilt \rtimes \pi_{1}(M)$ is a torsion-free finite index subgroup of $\tilw \rtimes \Gamma$. Since $\cu(\tilm,\tiln,\Gamma) /  (\tilt \rtimes \pi_{1}(M))$ is an aspherical closed manifold, 
$\tilt \rtimes \pi_{1}(M)$ is a Poincar\'e duality group. This verifies that $\tilw \rtimes \Gamma$ is a virtual Poincar\'e duality group.
\end{remark}

\bibliography{referencesforproject2}{}
\bibliographystyle{plain}

\leftline{\bf Author's address}

\noindent 
Raeyong Kim: 

Department of Mathematics, The Ohio State University, 

231 West 18th Avenue, Columbus, Ohio 43210-1174, United States.  

{\tt kimr@math.ohio-state.edu}

\end{document}